\documentclass[12pt]{amsart}
\usepackage{amscd, amssymb,latexsym,amsmath, amscd, amsmath}
\usepackage{amsthm}
\usepackage{amstext}
\usepackage{anysize}
\usepackage[sc]{mathpazo} 
\usepackage{enumitem}
\usepackage{pst-node}
\marginsize{2.5cm}{2.5cm}{2.5cm}{2.5cm}
\usepackage{graphics}
\usepackage{color}
\usepackage{tikz}
\usepackage{tikz-cd}
 \newtheorem{theorem}{Theorem}[section]
 
 \newtheorem{corollary}[theorem]{Corollary}
 \newtheorem{lemma}[theorem]{Lemma}
 \newtheorem{proposition}[theorem]{Proposition}

\newtheorem{remark}[theorem]{Remark}

\newtheorem{definition}[theorem]{Definition}

\newcommand{\GL}{{\rm GL}}
\newcommand{\GSpin}{{\rm GSpin}}
\newcommand{\SL}{{\rm SL}}
\newcommand{\Sp}{{\rm Sp}}
\newcommand{\SO}{{\rm SO}}
\newcommand{\C}{\mathbb{C}}
\newcommand{\Z}{\mathbb{Z}} 
\newcommand{\Q}{\mathbb{Q}} 
\newcommand{\Nm}{\mathbb{N}} 
\newcommand{\Irrep}{{\rm Irrep}}
\newcommand{\Ext}{{\rm Ext}}
\newcommand{\Hom}{{\rm Hom}}

\newcommand{\ind}{{\rm ind}}
\newcommand{\Ind}{{\rm Ind}}

\newcommand{\Gal}{{\rm Gal}}
\title{Restriction of representations of metaplectic $\GL_{2}(F)$ to tori}
\author{Shiv Prakash Patel}
\thanks{The first author was partially  supported by the Center for Advanced Studies in Mathematics, and by the Kreitman School of Advanced Graduate Studies at Ben-Gurion University of the Negev in Israel.}
\address{Department of Mathematics\\ Ben-Gurion University of the Negev\\ P.O.B. 653\\Be'er Sheva 8410501\\ Israel.}

\email{shivprakashpatel@gmail.com}

\author{Dipendra Prasad}

\address{School of Mathematics\\ Tata Institute of Fundamental Research \\ Homi Bhabha Road, Colaba \\Mumbai 400005\\ India.}

\email{dprasad@math.tifr.res.in}

\keywords{Branching laws, metaplectic groups, covering groups, restriction problems, Gross-Prasad conjectures, Heisenberg groups,
Representation theory of $p$-adic groups}

\date{\today}

\begin{document}

\begin{abstract}
Let $F$ be a non-Archimedean local field. 
We study the restriction of an irreducible admissible genuine representations of the two fold metaplectic cover $\widetilde{\GL}_{2}(F)$ of $\GL_{2}(F)$ 
to the inverse image in $\widetilde{\GL}_{2}(F)$ 
of a maximal torus in $\GL_{2}(F)$. 
\end{abstract}
\maketitle
\tableofcontents

\section{Introduction}
Let $F$ be a non-Archimedean local field. 
A well-known theorem due to J. Tunnell \cite{T83} for $p \neq 2$,  and H. Saito \cite{Sai93} in general, describes the restriction of an irreducible admissible representation of $\GL_{2}(F)$ to a maximal torus $E^{\times} \subset \GL_{2}(F)$, where $E$ is any maximal commutative semisimple subalgebra of $M_{2}(F)$.
One of the first conclusions about this restriction is that for any irreducible admissible representation $\pi$ of $\GL_{2}(F)$ and a character $\chi : E^{\times} \rightarrow \C^{\times}$ such that $\chi|_{F^{\times}}$ is the central character of $\pi$, then
\[
\dim \Hom_{E^{\times}} (\pi, \chi) \leq 1.
\]

If $\pi$ is a principal series representation of $\GL_{2}(F)$, or  $E = F \oplus F$ and $\pi$ not one dimensional,  
then we have
\[
\dim \Hom_{E^{\times}} (\pi, \chi) =1.
\]

This result on restriction of representations of $\GL_{2}(F)$ to maximal tori may be considered as the first case of branching laws from $\SO_{n+1}(F)$ to $\SO_{n}(F)$ which were formulated as conjectures by B. Gross and D. Prasad \cite{GP92}, and which were recently proved by Waldspurger and Moeglin-Waldspurger \cite{MW12}.

The aim of the present work is to initiate a similar study on restriction of representations of $\widetilde{\GL}_{2}(F)$, the metaplectic $\GL_{2}(F)$, which is a twofold cover of $\GL_{2}(F)$, to the inverse image $\tilde{E}^{\times}$ of $E^{\times}$ in $\widetilde{\GL}_{2}(F)$ where $E$ is any maximal commutative semisimple subalgebra of $M_{2}(F)$.

We will see that one of the crucial first steps, that of multiplicity one, is lost in the metaplectic case, although there is still finiteness, even boundedness of multiplicities by explicit constants.
It is hoped that metaplectic restriction problem will have some interest, and that this paper can serve as a first step.

The main theorem of this paper is the following.

\begin{theorem} \label{main theorem}
Let $E$ be any maximal commutative semisimple subalgebra of $M_{2}(F)$.
Let $\pi$ be an irreducible admissible genuine representation of $\widetilde{\GL}_{2}(F)$ with 
$\omega_{\pi}$ its central character (a character of $\tilde{F}^{\times 2}$).
If $E$ is a quadratic field extension of $F$, and $\pi$ is supercuspidal, assume moreover that $p$, the residue characteristic 
of $F$, is odd. Then
\begin{equation} \label{main relation}
\pi|_{\tilde{E}^{\times}} 
\subseteq \ind_{\tilde{F}^{\times 2}}^{\tilde{E}^{\times}} \omega_{\pi} = \sum_\sigma (\dim \sigma) \sigma
\end{equation}
where $\sigma$ runs 
over all irreducible genuine representations of ${\tilde{E}^{\times}} $ 
whose central character restricted to $\tilde{F}^{\times 2}$ is $\omega_\pi$.
Moreover if $\pi$ is an irreducible principal series representation,
 then we have "equality" in (\ref{main relation}).
\end{theorem}

\begin{remark} Of course we expect the theorem above to be true for $p =2$ too which we are not 
able to achieve here. 
In the spirit of {\rm dichotomy} of \cite{GP92}  we do not know if there is another representation $\pi'$ of $\widetilde{\GL}_2(F)$ 
such that the restriction to ${\tilde{E}^{\times}} $ 
of $\pi + \pi'$ achieves an ``equality'' up to finite error term in the above theorem,  
which is somehow accounted for by a twofold cover of $D^\times$ 
(containing ${\tilde{E}^{\times}} $!), where $D$ is the unique
quaternion division algebra over $F$.  
\end{remark} 

\begin{remark}\label{dimension} 
It will be seen later  that for an irreducible genuine representation $\sigma$ of $\tilde{E}^\times$, 
$\dim \sigma = |E^\times /F^\times E^{\times 2}|$, 
which for $p$ odd equals 2 by Corollary \ref{dim}, whereas for $p=2$, by remark \ref{p=2}, 
$\dim \sigma = 2\cdot 2^{deg(F/\Q_2)}$.
\end{remark}

In particular, we see that the multiplicity of an irreducible genuine representation $\sigma$ of $\tilde{E}^{\times}$ 
in an irreducible admissible genuine representation of $\widetilde{\GL}_{2}(F)$ is at most $\dim \sigma$.
The theorem turns out to be almost straightforward  to prove in the cases when the representation $\pi$ is either a 
principal series representation or $E = F \oplus F$.
The more difficult part --- something which we accomplish only for odd residue characteristic --- 
is to understand the restriction of an irreducible genuine supercuspidal representation $\pi$ of $\widetilde{\GL}_{2}(F)$ to $\tilde{E}^{\times}$ where $E/F$ is quadratic field extension and $E^{\times} \hookrightarrow \GL_{2}(F)$.
In this case,  when the residue characteristic is odd, we reduce the question on restriction from  $\widetilde{\GL}_{2}(F)$ to $\tilde{E}^{\times}$ to a question on restriction from $\SL_{2}(F)$ to $E^{1}$, where $E^{1}$ is the group of norm 1 elements of $E^{\times}$.
We shall do it in two steps.
\begin{enumerate}
\item Reduce the question on restriction from $\widetilde{\GL}_{2}(F)$ to $\tilde{E}^{\times}$ to a question on restriction from $\widetilde{\SL}_{2}(F)$ to $\tilde{E}^{1}$. 
This can be done for all residue characteristics.
\item When the residue characteristic is odd, using a correspondence (that we define in 
Section \ref{correspondence} using compact induction) between the set of isomorphism classes of 
irreducible genuine supercuspidal representation of $\widetilde{\SL}_{2}(F)$ and that of $\SL_{2}(F)$, we reduce the question of restriction from $\widetilde{\SL}_{2}(F)$ to $\tilde{E}^{1}$ to a question of restriction from $\SL_{2}(F)$ to $E^{1}$.
\end{enumerate}

In the second step, we need to restrict ourselves to the odd residue characteristic case because it is in 
this case when the metaplectic cover $\widetilde{\SL}_{2}(F)$ splits when it is restricted to a maximal compact subgroup of $\SL_{2}(F)$, and this splitting is used in executing the 2nd step.

We give a brief outline of the paper now.
In Section 2, we recall the twofold cover of $\GL_{2}(F)$ under consideration.
In Section 3, we describe the group structure on the inverse image in $\widetilde{\GL}_{2}(F)$ 
of  maximal tori in $\GL_{2}(F)$ in an explicit way.
The inverse images of tori may be called 'Heisenberg groups', which we discuss in some detail, proving some of its important properties and then describe their representations in Section 4.
In Section 5, we prove that the inverse images of tori are Heisenberg groups in the sense as defined in the earliar section, and then we describe their irreducible genuine representations.
In Section 6, the restriction of a genuine principal series representation to a non-split torus is considered.
In Section 7, the restriction of any irreducible admissible genuine representation to the split torus has been considered.
In Section 8, restricting ourselves to the case of odd residue characteristic, we define a correspondence between the irreducible genuine supercuspidal representation of $\widetilde{\SL}_{2}(F)$ and irreducible supercuspidal representation of $\SL_{2}(F)$.
In Section 9, we study the restriction of an irreducible supercuspidal representation of $\widetilde{\GL}_{2}(F)$ to a non-split torus.
We use the correspondence defined in Section 8 and transfer the question of this restriction to another question of restriction of a supercuspidal representation of $\SL_{2}(F)$ to $E^{1}$.

In closing the introduction, we mention that \cite{Pat15} which is the first author's thesis,  similar branching laws were 
considered from $\widetilde{\GL}_2(E)$ to $\GL_2(F)$ for $E$ a quadratic extension of $F$, which may be considered as branching laws from $\SO_4$ to $\SO_3$ in the
context of two-fold nonlinear covers. It may be added that the multiplicity formulae here 
involving $|E^\times /F^\times E^{\times 2}|$ (instead of 1 in \cite{GP92}), 
as contained in remark \ref{dimension}, is also the multiplicity obtained in \cite{Pat15} --- and 
could well be considered to be true more generally for branching laws for twofold metaplectic covers of $\GSpin_n(F)$ to the corresponding cover of 
$\GSpin_{n-1}(F)$. Both 
the work \cite{Pat15}, and this one, has the common feature with \cite {GP92}, that to get this `uniform multiplicity', 
we need to add {\it pure innerforms} of the smaller group in \cite {Pat15}, whereas here we need to add the 
contributions keeping the smaller group the same in this work. 
The methods in the two papers: \cite{Pat15} and this one, are quite different.

\section{Preliminaries}
Let $F$ be a non-Archimedean local field. 
The group $\SL_{2}(F)$ has a unique twofold cover (up to isomorphism), called the metaplectic cover of $\SL_{2}(F)$ denoted by $\widetilde{\SL}_{2}(F)$.
There are many in-equivalent twofold covers of $\GL_{2}(F)$ which extend the above twofold cover $\widetilde{\SL}_{2}(F)$ of $\SL_{2}(F)$.
In what follows, we fix a twofold covering of $\GL_{2}(F)$ as follows. 
Note that $\GL_{2}(F) \cong \SL_{2}(F) \rtimes F^{\times}$ where $F^{\times} \hookrightarrow \GL_{2}(F)$ as $a \mapsto \left( \begin{matrix} a & 0 \\ 0 & 1 \end{matrix} \right)$.
The action of $F^{\times}$ on $\SL_{2}(F)$ lifts to an action on $\widetilde{\SL}_{2}(F)$.
We fix the twofold cover $\widetilde{\GL}_{2}(F)$ of $\GL_{2}(F)$ as 
\[
\widetilde{\GL}_{2}(F) := \widetilde{\SL}_{2}(F) \rtimes F^{\times}
\]
and call it the metaplectic cover of $\GL_{2}(F)$.
We have a short exact sequence
\[
1 \rightarrow \mu_{2} \rightarrow \widetilde{\GL}_{2}(F) \rightarrow \GL_{2}(F) \rightarrow 1
\]
where $\mu_2 = \{ \pm 1 \}$.
This twofold cover $\widetilde{\GL}_{2}(F)$ of $\GL_{2}(F)$ is defined by a 2-cocycle, called Kubota cocycle
\[
\beta : \GL_{2}(F) \times \GL_{2}(F) \rightarrow \mu_{2}.
\]
We identify $\widetilde{\GL}_{2}(F)$ by $\GL_{2}(F) \times \mu_{2}$ as a set on which
 the group multiplication is defined using the cocycle $\beta$.
Let $B$ be the set of upper triangular matrices of $\GL_{2}(F)$. The restriction of $\beta$ to $B$ is given by 
\begin{equation} \label{cocycle}
 \beta \left( \left( \begin{matrix} a & x \\ 0 & b \end{matrix} \right), \left( \begin{matrix} c & y \\ 0 & d \end{matrix}  \right) \right) = (a,d)_{F}
\end{equation}
where $(\cdot, \cdot)_{F}$ denotes the quadratic Hilbert symbol of the field $F$.
In particular, if $A = \left( \begin{matrix} a & 0 \\ 0 & b \end{matrix} \right)$, $B = \left( \begin{matrix} c & 0 \\ 0 & d \end{matrix} \right)$ and $\tilde{A}, \tilde{B}$ are arbitrary lifts of $A, B$ to $\widetilde{\GL}_{2}(F)$, we have
\begin{equation}
[\tilde{A}, \tilde{B}] = (a,d)_{F} (c,b)_{F}.
\end{equation}
For a non-trivial character $\psi : F \rightarrow \C^{\times}$, let $\gamma(\psi)$ denote the 8-th root of unity associated to $\psi$ by A. Weil, called the Weil index.
For $a \in F^{\times}$, let $\psi_{a} : F \rightarrow \C^{\times}$ be the character of $F$ 
given by $\psi_{a}(x):= \psi(ax)$. 
Define
\[
\mu_{\psi}(a) := \dfrac{\gamma(\psi)}{\gamma(\psi_{a})}.
\]
It is known that
\begin{equation}
\mu_{\psi}(a) \mu_{\psi}(b) = (a,b)_{F} \mu_{\psi}(ab).
\end{equation}
Let $T_{0}$ be the diagonal split torus of $\SL_{2}(F)$.
Because of the commutation relation $(3)$,
the inverse image $\tilde{T}_{0}$ of $T_{0}$ is abelian.
For $a \in F^{\times}$, let $\underline{a}$ be the diagonal matrix $\left( \begin{matrix} a & 0 \\ 0 & a^{-1} \end{matrix} \right) \in \SL_{2}(F)$.
Because of $(4)$, the map $\tilde{T}_{0} \rightarrow \C^{\times}$ given by
\[
(\underline{a}, \epsilon) \mapsto \epsilon \mu_{\psi}(a)
\] 
defines a genuine character of $\tilde{T}_{0}$ where $\epsilon \in \mu_{2}$.

For any subset $X$ of $\GL_{2}(F)$, let $\tilde{X}$ be the full inverse image inside $\widetilde{\GL}_{2}(F)$ determined by the projection $\widetilde{\GL}_{2}(F) \rightarrow \GL_{2}(F)$.

Recall that $F^{\times}$ embedded diagonally as scalar matrices in $\GL_{2}(F)$ is the center of $\GL_{2}(F)$ and the covering $\widetilde{\GL}_{2}(F) \longrightarrow \GL_{2}(F)$ restricted to the center of $\GL_{2}(F)$  
is non-trivial. 
In fact, the cocycle is simply given by $\beta \left( \left( \begin{matrix} a & 0 \\ 0 & a \end{matrix} \right), \left( \begin{matrix} b& 0 \\ 0 & b \end{matrix} \right) \right) = (a,b)_{F}$ and hence the cover
\[ 1 \rightarrow \mu_{2} \rightarrow \tilde{F}^{\times} \rightarrow F^{\times} \rightarrow 1 \]
is non-trivial, although $\tilde{F}^{\times}$ is an abelian group.  
Note that $(a, \epsilon) \mapsto \epsilon \mu_{\psi}(a)$, where $\epsilon \in \mu_{2}$, defines a genuine character of $\tilde{F}^{\times}$.

\section{Group structure of inverse images of the tori}
Among the most important information about the covering $\widetilde{\GL}_{2}(F) 
\longrightarrow \GL_{2}(F)$ for us  
is the precise knowledge  about the group 
structure of the inverse image of the tori of $\GL_{2}(F)$ inside $\widetilde{\GL}_{2}(F)$.
First consider the case of split torus.

\begin{lemma} \label{group structure T}
Let $T$ be the diagonal torus of $\GL_{2}(F)$ and $T^{2} = \{ t^{2} : t \in T \}$.
The subgroup $\tilde{T}^2$ is the center of the group $\tilde{T}$. 
Let $Z = F^\times$ denote the center of $G$, then the subgroup $\tilde{Z} \tilde{T}^{2}$ is a maximal abelian subgroup of $\tilde{T}$.
\end{lemma}

\begin{proof}
From the commutation relation in $(3)$, it is clear that $\tilde{T}^{2}$ is contained in the center.
Let $x=diag(a,b) \in T$ and $\tilde{x} \in \tilde{T}$ be any lift of $x$. 
If $\tilde{x}$ is in the center of $\tilde{T}$ then we prove that $a, b \in F^{\times 2}$.
Suppose $\tilde{x}$ is in the center of $\tilde{T}$.
In particular, $\tilde{x}$ commutes with $diag(c,1)$ and $diag(1,d)$ for all $c, d \in F^{\times}$. 
By the commutation relation in $(3)$, 
this implies $(c,b)=1$ and $(a,d)=1$ for all $c, d \in F^{\times}$, i.e. $a, b \in F^{\times 2}$.
This proves that the center of $\tilde{T}$ is $\tilde{T}^{2}$.
Since $\tilde{Z}$ is abelian by the same commutation relation, $\tilde{Z} \tilde{T}^{2}$ is an abelian 
subgroup of $\widetilde{T}$.
We need to prove that it is a maximal abelian subgroup of $\widetilde{T}$.

Take $\tilde{x} \in \tilde{T}$ as above, and suppose it commutes with all the elements $\tilde{y} \in \tilde{Z} \tilde{T}^{2}$ where $y = diag(\alpha m^2, \alpha n^2)$ with $\alpha, m, n \in F^{\times}$.
By the commutation relation in $(3)$, we get that $(a, \alpha)= (b, \alpha)$, or in other words $(ab^{-1}, \alpha)=1$ for all $\alpha \in F^{\times}$ and hence $ab^{-1} \in F^{\times 2}$.
Thus $\tilde{x} \in \tilde{Z} \tilde{T}^{2}$.
\end{proof}

Now we consider the case of a non-split torus.
Let $E/F$ be a quadratic extension.
Let $E^{\times} \hookrightarrow \GL_{2}(F)$ be the non-split torus determined by the quadratic extension $E/F$.
We will continue to denote by  $T$ the diagonal torus of $\GL_{2}(F)$. 

Since the covering $1 \rightarrow \mu_{2} \rightarrow \tilde{F}^{\times} \rightarrow F^{\times} \rightarrow 1$ is non-trivial and $F^{\times} \hookrightarrow E^{\times}$, the cover 
\[ 1 \rightarrow \mu_{2} \rightarrow \tilde{E}^{\times} \rightarrow E^{\times} \rightarrow 1 \] 
is also non-trivial.
In fact, $\tilde{E}^{\times}$ 
is a non-abelian group.  The following lemma gives a more precise information on $\tilde{E}^{\times}$.

\begin{lemma} \cite[Proposition~0.1.5]{KP84} \label{KP commutator}
For $a,b \in E^{\times}$, let $\tilde{a}, \tilde{b}$ be any of the inverse images of $a,b$ in $\tilde{E}^{\times}$. 
The commutator $[\tilde{a}, \tilde{b}] \in \mu_2$ depends only on $a,b$, and is given by 
\[
[\tilde{a}, \tilde{b}] = (a,b)_{E} (\Nm a, \Nm b)_{F},
\] 
where $(\cdot, \cdot)_{E}$ and $(\cdot, \cdot)_{F}$ denote the quadratic Hilbert symbol of the field $E$ and $F$ respectively and $\Nm  : E^{\times} \rightarrow F^{\times}$ is the norm map for the field extension $E/F$.
\end{lemma}  

The following well-known relationship among the Hilbert symbols will be  very useful to us. 
We thank Adrian Vasiu for the proof below.
We will use this relationship on several occasions, sometimes without explicitly mentioning it.

\begin{lemma} \label{Hilbert symbol E/F}
Let $E/E$ be a finite extension of $p$-adic fields with $\mu_n \subset F^\times$.
For $a \in E^{\times}$ and $b \in F^{\times}$, we have
\[
(a,b)_{E} = (\Nm a, b)_{F},
\]
relating the $n$-th Hilbert symbols on $F$ and $E$.
\end{lemma}

\begin{proof}
Observe the following commutative diagram from the local class field theory:
\begin{equation}
\begin{tikzcd}
E^{\times} \arrow{r}{\sigma} \arrow{d}{\Nm } & \Gal(E^{ab}/E) \arrow{d}{res} 	\\
F^{\times} \arrow{r}{\sigma} & \Gal(F^{ab}/F)
\end{tikzcd}
\end{equation}
For $a \in E^\times$, or $a \in F^\times$, let $\sigma_a = \sigma(a)$ be the corresponding element of 
$\Gal(E^{ab}/E)$, or $\Gal(F^{ab}/F)$ as the case may be.  By definition,
\[
(a,b)_{E} = \frac{\sigma_{a}(b^{1/n})}{b^{1/n}},
\]
therefore, we have
\[
(\Nm a,b)_{E} = \frac{\sigma_{\Nm a}(b^{1/n})}{b^{1/n}}.
\]
By the above commutative diagram we have
\[
\sigma_{a}|_{F^{ab}} = \sigma_{\Nm a}
\]
and then the proof of the lemma follows.
\end{proof}

Now we describe some properties of the group $\tilde{E}^{\times}$.

\begin{lemma} \label{group structure E}
The subgroup $\tilde{E}^{\times 2}$ is contained in the center 
of $ \tilde{E}^{\times}$, 
and the subgroup $\tilde{F}^{\times} \tilde{E}^{\times 2}$ of $\tilde{E}^{\times}$ is a maximal abelian subgroup of $ \tilde{E}^{\times}$.
\end{lemma}

\begin{proof}
From the commutator relation in Lemma \ref{KP commutator}, 
it is clear that $\tilde{E}^{\times 2}$ is contained in the center of $\tilde{E}^{\times}$.
From the same commutator relation combined with Lemma \ref{Hilbert symbol E/F}, it follows that  $\tilde{F}^{\times}$ is abelian. Since 
$\tilde{E}^{\times 2}$ is contained in the center of $\tilde{E}^{\times}$, the subgroup $\tilde{F}^{\times} \tilde{E}^{\times 2}$ is abelian.

To prove that the subgroup $\tilde{F}^{\times} \tilde{E}^{\times 2}$ is maximal abelian, 
let $\tilde{a} \in \tilde{E}^{\times}$ commute with all the elements of $\tilde{F}^{\times} \tilde{E}^{\times 2}$. 
We need to prove that $\tilde{a} \in \tilde{F}^{\times} \tilde{E}^{\times 2}$.

Since $\tilde{E}^{\times 2}$ is contained in the center of $\tilde{E}^{\times}$, $[\tilde{a}, \tilde{b}]=1$ for all $\tilde{b} \in \tilde{F}^{\times} \tilde{E}^{\times 2}$ is equivalent to $[\tilde{a}, \tilde{b}]=1$ 
for all $\tilde{b} \in \tilde{F}^{\times}$. 
Using Lemma \ref{KP commutator} and Lemma \ref{Hilbert symbol E/F}, for $b \in F^{\times}$, we have:
\begin{eqnarray*}
 [\tilde{a}, \tilde{b}]=1 
& \iff  &   (a,b)_{E} (\Nm a, \Nm b)_{F} = 1 \\
& \iff & (\Nm a,b)_{F} (\Nm a, b^{2})_{F} =  1 \\
& \iff & (\Nm a, b)_{F} = 1.
\end{eqnarray*}

Therefore if $ [\tilde{a}, \tilde{b}]=1 $ for all $b \in F^\times$, then $\Nm a \in F^{\times 2}$, and this is possible only if $a \in F^{\times} E^{\times 2}$ 
(observe that since $e/\bar{e} = e^2/(e\bar{e}),  E^1 \subset F^{\times} E^{\times 2}$). 
This proves that $\tilde{F}^{\times} \tilde{E}^{\times 2}$ is a maximal abelian subgroup of $\tilde{E}^{\times}$.
\end{proof}

\begin{lemma} \label{center of tilde T}
The group $\tilde{E}^{\times 2}$ is equal to the center of $\tilde{E}^{\times}$. 
\end{lemma}

\begin{proof} We already know that $\tilde{E}^{\times 2}$ is contained in the center of $\tilde{E}^{\times}$, 
and that
$\tilde{F}^\times \cdot \tilde{E}^{\times 2}$ is a maximal abelian subgroup of $\tilde{E}^{\times}$. Let 
$f \in \tilde{F}^\times$  be in the centre of $\tilde{E}^{\times}$. It follows that, 
$$(f,\Nm e)_F=1, ~~~~~\forall e \in E^\times.$$
Since the Hilbert symbol is a non-degenerate bilinear form on $F^\times/F^{\times 2}$, and $\Nm E^\times$ is an index 
2 subgroup of $F^\times$, the orthogonal complement of $\Nm E^\times$ (this is defined to be the set of elements $a \in E^{\times}$ such that $(a,x)=1$ for all $x \in \Nm E^{\times}$) must contain $F^{\times 2}$ as a subgroup of index 2.

Suppose that $E= F(\sqrt{d})$. Observe the identity:

$$X^2 - dY^2 + dY^2 = X^2.$$
By the definition of the Hilbert symbol, this means that
$$(d,\Nm e)_F = 1,$$
for all $e \in E^\times$. Since $d$ is not a square in $F^\times$, it follows that the group generated by $F^{\times 2}$
and $d$ inside $F^\times$, i.e. $\langle F^{\times 2}, d \rangle$,  is the orthogonal complement on $\Nm E^\times$ for the Hilbert symbol of $F$.

It follows that if $f\in F^\times$ commutes with $E^{\times 2}$, then $f \in \langle F^{\times 2}, d \rangle$.
Since $d$ has a square root in $E^\times$ by definition, it follows that $\langle f, E^{\times 2} \rangle = E^{\times 2}$. This proves that $\tilde{E}^{\times 2}$ is equal to the center of $\tilde{E}^{\times}$. 
\end{proof}

\section{Heisenberg group and its representations}
The inverse images of tori (both split and non-split) of $\GL_{2}(F)$ inside $\widetilde{\GL}_{2}(F)$
are extensions of abelian groups by $\mu_2$, and  may be called `Heisenberg groups'. 
Although Heisenberg groups are omnipresent in representation theory, we do not know a convenient reference for our use, so we have preferred to define them and prove some of their key properties that will be used throughout the paper.

\begin{definition}[Heisenberg Group] \label{Heisenberg group}
A group $\Sigma$ with center $Z(\Sigma)$ with $\Sigma/Z(\Sigma)$ finite, and with $ [\Sigma,\Sigma] \cong \Z/p\Z \subset Z(\Sigma)$ for some prime $p$, will be called a Heisenberg group.
\end{definition}

For such a group $\Sigma$, the quotient $\Sigma/Z(\Sigma)$ is clearly an abelian group, and the commutator map defines a bilinear form
\begin{equation} \label{bilinear form}
B : \Sigma/Z(\Sigma) \times \Sigma/Z(\Sigma) \rightarrow \Z/p.
\end{equation}
i.e., for $e_{1}, e_{2} \in \Sigma/ Z(\Sigma)$, we have $B(e_{1}, e_{2}) := [\tilde{e}_{1}, \tilde{e}_{2}]$ where $\tilde{e}_{1}, \tilde{e}_{2}$ are arbitrary lifts of $e_{1}, e_{2}$ in $\Sigma$. 
For an abelian group $X$, let $\hat{X}$ denote the group of characters of $X$.
Define a homomorphism $f_{B} : \Sigma/Z(\Sigma) \rightarrow \widehat{\Sigma/Z(\Sigma)}$ as follows:
for all $a, e \in \Sigma/Z(\Sigma)$,
\[
f_{B}(a) (e) := \exp ^{\frac{2\pi i}{p} B(a,e)}.
\]

Observe that the homomorphism $f_{B}$ is injective (if $[\tilde{a},\tilde{e}]=1 \forall e \in \Sigma,$ 
then by definition, $\tilde{a} \in Z(\Sigma)$).
The bilinear form $B$ is said to be non-degenerate if the corresponding homomorphisms $f_{B}$ from $\Sigma/Z(\Sigma)$ to its character group is an isomorphism. 
In terms of the bilinear form $B$, a subgroup $A$ of $\Sigma$ is abelian if and only if the bilinear form $B$ on $AZ(\Sigma)/Z(\Sigma)$ is identically zero. 
The subgroups of $\Sigma/Z(\Sigma)$ on which the bilinear form $B$ is identically zero are called isotropic subgroups.
It follows that a subgroup $A$ of $\Sigma$ containing $Z(\Sigma)$ 
is maximal abelian if and only if its image in $\Sigma/Z(\Sigma)$  is maximal isotropic. 
If $A$ is isotropic, the natural map from $\Sigma/A$ to the set of characters $\hat{A}$ of $A$ is surjective and if $A$ is maximal isotropic, this map is an isomorphism.
Note that if $A$ is an abelian subgroup of $\Sigma$ then the subgroup of $\Sigma$ generated by $A$ and $Z(\Sigma)$ is also abelian.
It follows that if $A$ is a maximal abelian subgroup of $\Sigma$ then $A$ necessarily contains the center $Z(\Sigma)$ of $\Sigma$.

\begin{lemma} \label{index relation in H-group}
Let $\Sigma$ be a Heisenberg group in the sense of Definition \ref{Heisenberg group} for which the corresponding bilinear form $B$ given in (\ref{bilinear form}) is non-degenerate. Let $A$ be a maximal abelian subgroup of $\Sigma$, then
\[
[A : Z(\Sigma)]^{2} =[\Sigma : Z(\Sigma)].
\]
\end{lemma}

\begin{proof}
Because of the non-degeneracy of the bilinear form $B$, the natural map $\Sigma/A \rightarrow \widehat{A/Z(\Sigma)}$ is an isomorphism and hence $[\Sigma : A] = \mid \widehat{A/Z(\Sigma)} \mid$.
The lemma follows from the obvious relations:
\begin{center}
$[\Sigma : A]\cdot [A : Z(\Sigma)] =[\Sigma : Z(\Sigma)] $ 
and $[A : Z(\Sigma)] = \mid \widehat{A/Z(\Sigma)} \mid$.
\end{center}
\end{proof}

A key property of Heisenberg groups that we will use is the following.
Let $A$ be a maximal abelian subgroup of $\Sigma$ (thus containing $Z(\Sigma)$). 
Then $\Sigma/A$ is an abelian group and naturally acts on $\hat{A}$, the set of characters of $A$ which are non-trivial on 
$\Z/p \subset Z(\Sigma)$.
This action of $\Sigma/A$ on $\hat{A}$ is faithful, i.e., $\chi^e \not = \chi$ for $e \not = 1$ in $\Sigma/A$. Equivalently,
 if $e \in \Sigma \setminus A$, 
$\chi \in \hat{A}$ is non-trivial on $\Z/p \subset Z(\Sigma)$, then there exists an $a \in A$ such that $\chi(eae^{-1}) \neq \chi(a)$, i.e. $\chi(eae^{-1} a^{-1}) \neq 1$.
By the maximality of the abelian subgroup $A$ inside $\Sigma$, given $e \in \Sigma \setminus A$ there exists an $a \in A$ 
such that $eae^{-1}a^{-1} \neq 1$. 
Since 
$eae^{-1}a^{-1} \in \Z/p$ and any nontrivial element of $\Z/p$ generates $\Z/p$, it follows that for any nontrivial
character $\chi$ of $A$ which is nontrivial on $\Z/p$, $\chi(eae^{-1}a^{-1}) \not = 1$ for any $e \in \Sigma \setminus A$. 

\begin{definition}
An irreducible 
 representation of a subgroup of a Heisenberg group $\Sigma$ which contains $Z(\Sigma)$, the center of $\Sigma$, is called genuine if its restriction to $\Z /p \subset Z(\Sigma)$ is a non-trivial character of $\Z /p$.
\end{definition}

\begin{proposition} \label{reps of H-groups}
Let $\Sigma_{1}$ be a maximal abelian subgroup of a Heisenberg group $\Sigma$ (such a subgroup is automatically normal and contains the center $Z(\Sigma)$ of $\Sigma$).
Then
\begin{enumerate}
\item Any irreducible genuine representation of $\Sigma$  is obtained by inducing a genuine character of $\Sigma_{1}$.
\item Conversely, $\Ind_{\Sigma_{1}}^{\Sigma} \lambda$ is irreducible for any character $\lambda:\Sigma_{1} \rightarrow \C^\times$ with
$\lambda|_{\Z/p} \not = 1$.
\item For characters $\lambda_1,\lambda_2: \Sigma_1 \rightarrow \C^\times$, we have 
$\Ind_{\Sigma_{1}}^{\Sigma} \lambda_{1} \cong \Ind_{\Sigma_{1}}^{\Sigma} \lambda_{2}$ if and only if $\lambda_{1} = \lambda^s_{2}$ 
for some  $s \in \Sigma$.
\item The restriction of an irreducible genuine representation of $\Sigma$ to $\Sigma_1$ is a sum of distinct genuine characters $\lambda^s : \Sigma_1 \rightarrow \C^\times$,  for 
$s \in \Sigma/\Sigma_1$.
\item The restriction of an irreducible genuine representation $\sigma$ of $\Sigma$ to $\Sigma_{1}$ is sum of all genuine characters of $\Sigma_{1}$ with multiplicity 1 whose restriction to $Z(\Sigma)$ is $\omega_{\sigma}$, the central character of $\sigma$, i.e., 
%$\sigma|_{\Sigma_1}= \ind_{Z(\Sigma)}^{\Sigma_1} \omega_{\sigma}$. 
\begin{center}
$\sigma|_{\Sigma_1}= \ind_{Z(\Sigma)}^{\Sigma_1} \omega_{\sigma}$.
\end{center}
\end{enumerate}
\end{proposition}

\begin{proof}
Let $\pi$ be any irreducible genuine representation of $\Sigma$, and let $\lambda$ be a character of $\Sigma_{1}$ which appears in $\pi$ restricted to $\Sigma_{1}$.
By Clifford theory, if the action of $\Sigma/\Sigma_{1}$ on genuine characters of $\Sigma_{1}$ is faithful, then $\pi \cong \Ind_{\Sigma_{1}}^{\Sigma} \lambda$ for any character $\lambda$ of $\Sigma_{1}$ appearing in $\pi$. 
By the basic property of Heisenberg groups established already, we do know that the action of $\Sigma/\Sigma_{1}$ on genuine characters of $\Sigma_{1}$ is faithful proving part $(1)$ and $(2)$ of the proposition. 
Part (3) and (4) are clear as well.
For part (5), it is clear that $\ind_{Z(\Sigma)}^{\Sigma_1} \omega_{\sigma}$ is contained in $\sigma$ restricted to $\Sigma_1$. 
To prove equality, it suffices to prove that the two representations have the same dimension. 
Observe that $\dim \ind_{Z(\Sigma)}^{\Sigma_1} \omega_{\sigma} = \#(\Sigma_{1}/Z(\Sigma))$ where as by part (1), $\dim \sigma = \#(\Sigma/\Sigma_{1})$. 
However by Lemma \ref{index relation in H-group}, $\#(\Sigma/Z(\Sigma))= \#(\Sigma/\Sigma_{1})^2$. 
Therefore, $\#(\Sigma/\Sigma_{1}) = \#(\Sigma/Z(\Sigma))$.
\end{proof}

\begin{corollary} \label{restriction to max-abelian}
Suppose $\Sigma_{2}$ is an abelian subgroup of a Heisenberg group $\Sigma$ with $\Sigma_{2} Z(\Sigma)$ a maximal abelian in $\Sigma$. 
Then the restriction of an irreducible representation $\sigma$ of $\Sigma$ to $\Sigma_{2}$ is 
\begin{center}
$\ind_{{\Sigma_{2}} \cap Z(\Sigma)}^{\Sigma_{2}} \omega_{\sigma}$.
\end{center} 
\end{corollary}

\section{Representations of $\tilde{T}$ and $\tilde{E}^{\times}$}
To describe the representations of $\tilde{T}$ and $\tilde{E}^{\times}$ we first note the following fact.

\begin{lemma}
The groups $\tilde{T}$ and $\tilde{E}^{\times}$ are Heisenberg groups in the sense of definition \ref{Heisenberg group} with 
$[\tilde{T},\tilde{T}] = \Z/2 = [\tilde{E}^\times, \tilde{E}^\times]$.
Moreover, the bilinear forms corresponding to the Heisenberg groups $\tilde{T}$ and $\tilde{E}^{\times}$, as defined in (\ref{bilinear form}), are non-degenerate.
\end{lemma} 

\begin{proof} Since $T$ is abelian, $[\tilde{T},\tilde{T}] \subset \Z/2 $, and similarly $E^\times$ being abelian,
$ [\tilde{E}^\times, \tilde{E}^\times] \subset \Z/2$. Since we know that $\tilde{T}$ as well as $\tilde{E}^\times$ are non-abelian
(because we know they have a proper maximal abelian subgroup!), it follows that $[\tilde{T},\tilde{T}] = \Z/2 = [\tilde{E}^\times, \tilde{E}^\times]$.

Now we prove the non-degeneracy of the bilinear forms for $\tilde{T}$ and $\tilde{E}^{\times}$.
For the rest of the proof, write $\Sigma$ for either of the two Heisenberg groups $\tilde{T}$ and $\tilde{E}^{\times}$.
We need to prove that the homomorphism
\begin{equation} \label{bijection1}
\Sigma/ Z(\Sigma) \rightarrow \widehat{\Sigma/ Z(\Sigma)}
\end{equation}
defined by $e \mapsto (x \mapsto [\tilde{e}, \tilde{x}])$, where $\tilde{e}$ and $\tilde{x}$ are arbitrary lift of $e$ and $x$ in $\Sigma$, is an isomorphism.
Recall that the indices $[\tilde{T} : Z(\tilde{T})] = [\tilde{T} : \tilde{T}^{2}] = [F^{\times} : F^{\times 2}]^{2}$ and $[\tilde{E}^{\times} : \tilde{E}^{\times 2}] = [E^{\times} : E^{\times 2}]$ are finite.
Since $\Sigma/ Z(\Sigma)$ is a  finite abelian group, the cardinality of $\Sigma/ Z(\Sigma)$ and $\widehat{(\Sigma/ Z(\Sigma))}$ are the same.
Since the map (\ref{bijection1}) is known to be injective, it is also surjective.
\end{proof}

\begin{corollary}\label{dim}
For a maximal abelian subgroup $A$ of $\tilde{E}^{\times }$
(necessarily containing $\tilde{E}^{\times 2}$), 
we have $[\tilde{E}^{\times} : \tilde{E}^{\times 2}] = [\tilde{E}^{\times} : A]^{2}$. In particular, $[E^{\times} : E^{\times 2}] = [E^{\times} : F^{\times} E^{\times 2}]^{2}$.
\end{corollary}

\begin{remark} \label{p=2}
Assume $p = 2$ for this remark.
It is known that $[E^{\times} : E^{\times 2}] = 4 \cdot 2^{deg(E/\Q_{2})}$ and $[F^{\times} : F^{\times 2}] = 4 \cdot 2^{deg(F/\Q_{2})}$. 
For $E = F(\sqrt{d})$ we have $F^{\times} \cap E^{\times 2} = F^{\times 2} \cup d F^{\times 2}$ and hence $[F^{\times} E^{\times 2} : E^{\times 2}] = [F^{\times} : F^{\times} \cap E^{\times 2}] = \frac{1}{2} [F^{\times} : F^{\times 2}]$.
Then the obvious identity $[E^{\times} : E^{\times 2}] = [E^{\times} : F^{\times} E^{\times 2}] [F^{\times} E^{\times 2} : E^{\times 2}]$ confirms the conclusion in the above corollary.
\end{remark}
From the Proposition \ref{reps of H-groups}, Lemma \ref{group structure T} and Lemma \ref{group structure E}, we deduce the following

\begin{proposition} \label{irrep of tilde T and tilde E}
\begin{enumerate}
\item Up to isomorphism, an irreducible genuine representation $\sigma$ of $\tilde{T}$ is determined by its central character $\omega_\sigma$ (a character of $\tilde{T}^2$). If $\sigma$ is an irreducible genuine representation of $\tilde{T}$ with central character $\omega_{\sigma}$, 
 $\lambda$ a character of $\tilde{Z} \tilde{T}^{2}$ which agrees with $\omega_{\sigma}$ when restricted to 
$\tilde{T}^2$, the center of $\tilde{T}$, then $\sigma \cong \Ind_{\tilde{Z} \tilde{T}^{2}}^{\tilde{T}} \lambda$.

\item  Up to isomorphism, an irreducible genuine representation of $\tilde{E}^{\times}$ is determined by its central character. If $\sigma$ is an irreducible genuine representation of $\tilde{E}^{\times}$ with central character $\omega_{\sigma}$,  $\lambda$ a character of $\tilde{F}^{\times} \tilde{E}^{\times 2}$ which agrees with $\omega_{\sigma}$
on $\tilde{E}^{\times2}$, then $\sigma \cong \Ind_{\tilde{F}^{\times} \tilde{E}^{\times 2}}^{\tilde{E}^{\times}} \lambda$.

\end{enumerate}
\end{proposition}

\section{Restriction of principal series representations}

In this section, we study the restriction of a genuine principal series representation of 
$\widetilde{\GL}_{2}(F)$ 
to the subgroup  $\tilde{E}^{\times}$ for $E$ a quadratic field extension of $F$.

We first recall the notion of a principal series representation.
Let $T, B$ and $N$ be respectively the group of 
 diagonal matrices, upper triangular matrices and upper triangular unipotent matrices in $G=\GL_{2}(F)$,
and $\tilde{T}, \tilde{B}$ and $\tilde{N}$ their inverse images in 
the twofold cover $\tilde{G}=\widetilde{\GL}_{2}(F)$.

From the cocycle formula (\ref{cocycle}), it is clear that $\tilde{N} \cong N \times \mu_{2}$ and we identify $N$ with $N \times \{1\}$ in $\widetilde{\GL}_{2}(F)$. 
One has $\tilde{B} = \tilde{T} N$.

Let $\tau$ be a genuine irreducible representation of $\tilde{T}$. 
Take the inflation of the representation $\tau$ to a representation of $\tilde{B}$ by the quotient map $\tilde{B} \rightarrow \tilde{B}/N \cong \tilde{T}$ and denote this by the same letter $\tau$.
The induced representation $\Ind_{\tilde{B}}^{\tilde{G}} \tau$ of $\tilde{G}$ is  called a principal series representation 
of $\widetilde{\GL}_{2}(F)$.

Since $\tilde{T}$ is a Heisenberg group, its genuine irreducible representations are determined (up to isomorphism) by its central character, i.e. a genuine character of $\tilde{T}^{2}$.
A character $\chi'$ of $F^{\times 2}$ can be considered as the restriction of a character $\chi$ of $F^\times$ with $\chi'(a^2) = \chi^2(a)$. Two characters $\chi_1$ and $\chi_2$ of $F^\times $ define the same character of $F^{\times 2}$ if and only if $\chi_1^2 = \chi_2^2$. 
Thus to a principal series representation of  $\widetilde{\GL}_{2}(F)$, there is a naturally associated principal series
representation $\chi_1^2 \times \chi_2^2$ of $\GL_2(F)$. 

It is a theorem due to Moen \cite{Moe89} that a principal series representation of  $\widetilde{\GL}_{2}(F)$ 
is reducible if and only if 
 the  associated principal series
representation $\chi_1^2 \times \chi_2^2$ of $\GL_2(F)$ is reducible.

We now study the restriction of $\pi$ to the subgroup $\tilde{E}^{\times}$. 

\begin{lemma}
Let $\pi= \Ind_{\tilde{B}}^{\tilde{G}} \tau$ be a genuine principal series representation of $\widetilde{\GL}_{2}(F)$ with central character $\omega_\pi$ (a character of $\tilde{F}^{\times 2}
$), i.e., $\omega_\tau|_{\tilde{F}^{\times 2}} = \omega_\pi$. 
Then the restriction of $\pi$ to $\tilde{E}^{\times}$ is 
\[
\pi|_{\tilde{E}^{\times}} \cong \Ind_{\tilde{F}^{\times 2}}^{\tilde{E}^{\times}} (\omega_{\tau}) \cong \bigoplus_{\sigma \in \Irrep_{\omega_{\pi}} (\tilde{E}^{\times})} (\dim \sigma) \sigma,
\]
where $\Irrep_{\omega_{\pi}}(\tilde{E}^{\times})$ denotes the set of isomorphism classes of irreducible genuine representations $\sigma$ of $\tilde{E}^{\times}$ such that $\omega_{\sigma}|_{\tilde{F}^{\times 2}} = \omega_{\pi}$. (Recall that by Lemma \ref{center of tilde T}, center of $\tilde{E}^\times$ is $\tilde{E}^{\times 2}$).
\end{lemma}

\begin{proof}
Note that the natural (right) action of $\tilde{E}^{\times}$ on $\tilde{B} \backslash \tilde{G}= \mathbb{P}^{1}(F)$ is transitive, i.e. there is only one orbit.
By Mackey theory, $\pi|_{\tilde{E}^{\times}} \cong \Ind_{\tilde{B} \cap \tilde{E}^{\times}}^{\tilde{E}^{\times}} (\tau|_{\tilde{B} \cap \tilde{E}^{\times}})$.
Since $\tilde{B} \cap \tilde{E}^{\times} = \tilde{F}^\times$, we have 
$\pi|_{\tilde{E}^{\times}} \cong \Ind_{\tilde{F}^\times}^{\tilde{E}^{\times}} (\tau|_{\tilde{F}^\times})$.
From Corollary \ref{restriction to max-abelian}, 
\begin{center}
$\tau|_{\tilde{F}^\times} = \Ind_{\tilde{F}^\times \cap Z(\tilde{T})}^{\tilde{F}^\times} \left( \omega_{\tau}|_{\tilde{F}^\times \cap Z(\tilde{T})} \right) = \Ind_{\tilde{F}^{\times2}}^{\tilde{F}^\times} (\omega_{\pi})$.
\end{center}
Therefore, $\pi|_{\tilde{E}^{\times}} \cong \Ind_{\tilde{F}^{\times 2}}^{\tilde{E}^{\times}} (\omega_{\pi})$.
Since $\tilde{E}^{\times}$ is a group which is compact modulo the center, the 2nd isomorphism in the 
assertion of  the lemma follows from Frobenius reciprocity.
\end{proof}

\begin{corollary} \label{sum of multiplicities} Let $\sigma$ be an irreducible genuine representation of $\tilde{E}^{\times}$.
\begin{enumerate}
\item Let $\pi = \Ind_{\tilde{B}}^{\tilde{G}} \tau$ be an irreducible genuine principal series such that $\omega_{\sigma}|_{\tilde{F}^{\times 2}} = \omega_{\pi}$. Then
\[
\dim \Hom_{\tilde{E}^{\times}}(\pi, \sigma) = \dim \sigma = [E^{\times} : F^{\times} E^{\times 2}].
\]
\item
 Let $\pi$ and $\pi'$ be the two sub-quotients of a genuine reducible principal series  representation $\Ind_{\tilde{B}}^{\tilde{G}} \tau$ such that $\omega_{\sigma}|_{\tilde{F}^{\times 2}} = \omega_{\pi}$. Then for a quadratic extension $E$ of $F$,
\[
\dim \Hom_{\tilde{E}^{\times}}(\pi, \sigma) + \dim \Hom_{\tilde{E}^{\times}}(\pi', \sigma) = [E^{\times} : F^{\times} E^{\times 2}].
\]
\end{enumerate}
\end{corollary}
\section{Restriction to split torus}
In this section, we study the restriction of an irreducible admissible genuine representation of $\widetilde{\GL}_{2}(F)$ to $\tilde{T}$. 
We will be utilizing the Kirillov model for the representations of $\widetilde{\GL}_{2}(F)$ \cite[~Section~3]{GPS80}.

Let $\psi : N \cong F \rightarrow \C^{\times}$ be a non-trivial character.
Recall that the Kirillov model is an injective map $\mathtt{K} :\pi \rightarrow C^{\infty}(F^{\times}, \pi_{N, \psi})$ such that 
the action of $\tilde{B}$ on $\pi$ is explicitly realized on the image of $\mathtt{K}$.
The subspace $\mathcal{S}(F^{\times}, \pi_{N, \psi})$ consisting of the functions which have compact support is contained in $\mathtt{K}(\pi)$, the image of $\mathtt{K}$,  which is $\tilde{B}$-stable.
Moreover, as a $\tilde{B}$-module we have the following short exact sequence
\[
0 \rightarrow \mathcal{S}(F^{\times}, \pi_{N, \psi}) \rightarrow \mathtt{K}(\pi) \rightarrow \pi_{N} \rightarrow 0.
\]
Notice that $\pi_{N, \psi}$ is a $\tilde{Z}$-module on which 
$N$ acts by the character $\psi$.

\begin{proposition}
Let $\pi$ be an irreducible admissible genuine representation of $\widetilde{\GL}_{2}(F)$.
\begin{enumerate}
\item As a $\tilde{B}$-module, 
\[
\mathcal{S}(F^{\times}, \pi_{N, \psi}) \cong \ind_{\tilde{Z}N}^{\tilde{B}} (\pi_{N, \psi}).
\]
\item 
As a $\tilde{T}$-module,
\[
\mathcal{S}(F^{\times}, \pi_{N, \psi})  \cong \ind_{\tilde{Z}}^{\tilde{T}} (\pi_{N, \psi}).
\]
\end{enumerate}
\end{proposition}

\begin{proof}
The first part is part of Kirillov theory, see \cite[~Section~3]{GPS80}., and the second is a consequence of Mackey theory.
\end{proof}

\begin{corollary}
Let $\pi$ be an irreducible genuine supercuspidal representation of $\widetilde{\GL}_{2}(F)$ and $\sigma$ an irreducible genuine representation of $\tilde{T}$ with $\omega_{\pi} = \omega_{\sigma}|_{\tilde{Z}^2}$. Then
\[
\Hom_{\tilde{T}} (\pi, \sigma) \cong \Hom_{\tilde{Z}} (\pi_{N, \psi}, \sigma) \cong \Hom_{\tilde{Z}^2}(\pi_{N,\psi}, \omega_\sigma).
\]
In particular,
\[
\dim \Hom_{\tilde{T}} (\pi, \sigma) = \dim \pi_{N, \psi} = \mid \Omega(\pi, \psi) \mid
\]
which is independent of the choice of the additive character $\psi$ of $F$, and where $$\Omega(\pi, \psi) = \{\omega: \tilde{F}^\times \rightarrow \C^\times \mid \pi_{N,\psi} {\rm~~contains~~} \omega_\pi \}.$$ 
\end{corollary}

\begin{proof}
Since $\pi_{N} =0$, we have $\pi|_{\tilde{T}} \cong \mathcal{S}(F^{\times}, \pi_{N, \psi})|_{\tilde{T}} \cong \ind_{\tilde{Z}}^{\tilde{T}} (\pi_{N, \psi})$.  From Corollary \ref{restriction to max-abelian}, 
$$\sigma|_{\tilde{Z}} = {\rm Ind}_{\tilde{Z}^2}^{\tilde{Z}} (\omega_\sigma).$$

By Frobenius reciprocity,
$$\Hom_{\tilde{T}} (\pi, \sigma) \cong \Hom_{\tilde{Z}} (\pi_{N, \psi}, \sigma) \cong \Hom_{\tilde{Z}^2}(\pi_{N,\psi}, \omega_\sigma),$$
which proves the first isomorphism contained in the  corollary.

The second isomorphism contained in the  corollary  follows since the multiplicity with which any 
genuine character of $\tilde{Z}$ is contained in $\pi_{N, \psi}$ is exactly one \cite[Theorem~4.1]{GHPS79}, and every character in $\Omega(\pi, \psi)$ restricted to $\tilde{Z}^{2}$ naturally equals  $\omega_\pi= \omega_{\sigma}|_{\tilde{Z}^{2}}$. \end{proof}

\begin{corollary} \label{pi and pi' for split torus}
Let $\pi$ be an irreducible admissible genuine supercuspidal representation of $\widetilde{\GL}_{2}(F)$ and $\psi$ a non-trivial additive character of $F$.
Let $\pi'$ be another finite length  genuine supercuspidal representation of $\widetilde{\GL}_{2}(F)$ with the same central character as  ${\pi}$,  satisfying
\[
\Omega(\pi, \psi)  \sqcup \Omega(\pi', \psi) = \Omega(\omega_{\pi}),
\] a disjoint union of sets, i.e., any character $\chi$ of $\tilde{F}^\times$ 
whose restriction to $\tilde{F}^{\times 2}$ is $\omega_\pi$ appears in
exactly one of $\pi_{N,\psi}$ or $\pi'_{N,\psi}$. (As recalled earlier, 
by \cite[Theorem~4.1]{GHPS79} every character of $\tilde{F}^\times$ appearing in $\pi_{N,\psi}$ appears with multiplicity at most 1.)
 
Then we have 
\[
\pi|_{\tilde{T}} \oplus \pi'|_{\tilde{T}} \cong \bigoplus_{\sigma \in \Irrep_{\omega_{\pi}}(\tilde{T})}(\dim \sigma)  \sigma 
\]
\end{corollary}

For a principal series representations $\Ind_{\tilde{B}}^{\tilde{G}} \tau$ 
where we use un-normalized induction, instead of using the Kirillov theory, we directly use the action of $\widetilde{T}$
on the principal series representation by the geometeric nature of this action, and by Mackey theory get the following
exact sequence of $\tilde{T}$-modules: 
\begin{equation} \label{ses for ps resticted to tilde T}
0 \rightarrow \ind_{\tilde{Z}}^{\tilde{T}} (\tau|_{\tilde{Z}})  
\rightarrow \pi \rightarrow \tau + \tau^\omega \rightarrow 0.
\end{equation}
For an irreducible genuine representation $\sigma$ of $\tilde{T}$, the functor $\Hom_{\tilde{T}} (-, \sigma)$ when applied to the short exact sequence in (\ref{ses for ps resticted to tilde T}) results in the following long exact sequence:
\[
0 \rightarrow \Hom_{\tilde{T}} (\tau + \tau^{w}, \sigma) \rightarrow \Hom_{\tilde{T}} (\pi, \sigma) \rightarrow \Hom_{\tilde{T}} (\ind_{\tilde{Z}}^{\tilde{T}} (\tau|_{\tilde{Z}}), \sigma) \rightarrow \Ext_{\tilde{T}}^{1} (\tau + \tau^{w}, \sigma) 
\]

From the Lemma \ref{extension lemma for tilde T} below, it follows that all representations of $\tilde{T}$ except $\tau $ and $\tau^\omega$ appear with the multiplicity
with which it appears in $\ind_{\tilde{Z}}^{\tilde{T}} (\tau|_{\tilde{Z}})  $ which is
\[
\dim \Hom_{\tilde{T}} (\pi, \sigma) = [F^{\times} : F^{\times 2}].
\]

On the other hand, if $\pi$ is an irreducible principal series, it can also be expressed as 
$\Ind_{\tilde{B}}^{\tilde{G}} (\delta \cdot \tau^\omega)$ (as un-normalized induction). 
This realization of the principal series $\pi$ gives us the following long exact sequence of $\tilde{T}$-modules:
\begin{equation} 
0 \rightarrow \ind_{\tilde{Z}}^{\tilde{T}} (\delta \tau^{w}|_{\tilde{Z}})  = \ind_{\tilde{Z}}^{\tilde{T}} ( \tau^{w}|_{\tilde{Z}})  
\rightarrow \pi \rightarrow \delta^{-1}\tau + \delta \tau^\omega \rightarrow 0.
\end{equation}

Using this form of principal series, it follows by the same reasoning as above, 
that  all representations of $\tilde{T}$ except $\delta \cdot \tau^\omega $ and $\delta^{-1} \cdot \tau$ appear with the multiplicity
with which it appears in $\ind_{\tilde{Z}}^{\tilde{T}} (\tau|_{\tilde{Z}})  $ which is $[F^{\times} : F^{\times 2}]$.

Next we  observe the following Lemma.

\begin{lemma} For an irreducible  principal series representation $\pi = \Ind_{\tilde{B}}^{\tilde{G}} \tau$,
and an irreducible genuine representation $\sigma$ of $\tilde{T}$, either 
$\sigma \not \in \{ \tau, \tau^{w} \}$ or $\sigma \not \in \{ \delta \tau^{w}, \delta^{-1} \tau \}$.
\end{lemma}
\begin{proof} 
From Proposition \ref{irrep of tilde T and tilde E}(1), an irreducible representation $\sigma$ of
$\tilde{T}$ is determined by its central character $\omega_\sigma$ (a character of $\tilde{T}^2$). 
Therefore, if $\sigma \in \{ \tau, \tau^{w} \} \cap  \{ \delta \tau^{w}, \delta^{-1} \tau \}$, 
the central character of a representation in $ \{ \tau, \tau^{w} \}$ must be that of one in 
 $\{ \delta \tau^{w}, \delta^{-1} \tau \}$. 
Central character being a character of $\tilde{T}^2$, write 
the central character of $\tau$ as $(\chi_1,\chi_2)$ where $\chi_1,\chi_2: F^{\times 2} \rightarrow \C^\times$.
Thus the central character of $\tau^w$ is $(\chi_2,\chi_1)$, that of 
$\delta \tau^{w}$ is $(\nu^{1/2}\chi_2, \nu^{-1/2}\chi_1)$, 
and that of  $\delta^{-1} \tau $ is $(\nu^{-1/2}\chi_1, \nu^{1/2}\chi_2)$.
It follows that the set $\{ \tau, \tau^{w} \} \cap  \{ \delta \tau^{w}, \delta^{-1} \tau \}$ is nonempty 
only when (as characters of $F^{\times 2}$)
$$(\chi_1/\chi_2, \chi_2/\chi_1) = (\nu^{1/2},\nu^{-1/2}), 
{\rm ~~or~~}(\nu^{-1/2},\nu^{1/2}).$$ 
 
As recalled earlier, by \cite{Moe89}, these are exactly the conditions for reducibility of the
genuine principal series representation 
$\Ind_{\tilde{B}}^{\tilde{G}} \tau$,
which we are excluding, thus the proof of the lemma is completed. 
\end{proof}

We summarize our analysis on irreducible principal series representations in the following proposition.

\begin{proposition} For an irreducible  principal series representation $\pi = \Ind_{\tilde{B}}^{\tilde{G}} \tau$,
and an irreducible genuine representation $\sigma$ of $\tilde{T}$ such that $\omega_\pi = \omega_\sigma|_{\tilde{F}^{\times 2}}$,
\[
\dim \Hom_{\tilde{T}} (\pi, \sigma) = [F^{\times} : F^{\times 2}].
\]
\end{proposition}

\begin{lemma} \label{extension lemma for tilde T}
Let $\sigma_{1}$ and $\sigma_{2}$ be two irreducible genuine representations of $\tilde{T}$ with 
$\sigma_{1}|_{\tilde{Z}} = \sigma_{2}|_{\tilde{Z}}$.
Then
\[
\dim \Hom_{\tilde{T}} (\sigma_{1}, \sigma_{2}) = \dim \Ext^{1}_{\tilde{T}} (\sigma_{1}, \sigma_{2}),
\]
and $\Ext^{i}_{\tilde{T}} (\sigma_{1}, \sigma_{2}) =0$ for $i \geq 2$,
where $\Ext^{i}_{\tilde{T}}$ is calculated in the category of representations of $\tilde{T}$ with a
 given central character  of $\tilde{Z}^2$ which is $\sigma_{1}|_{\tilde{Z}^2} = \sigma_{2}|_{\tilde{Z}^2}$.
\end{lemma}

\begin{proof}
We know that any irreducible genuine representation of $\tilde{T}$ is obtained as an induced
representation from a genuine character of the finite index subgroup $\tilde{Z} \tilde{T}^{2}$,
say $\sigma_{2} = \ind_{\tilde{Z} \tilde{T}^{2}}^{\tilde{T}} (\chi_{2})$.
By Frobenius reciprocity, for all $i \geq 0$ we have
\[
\Ext_{\tilde{T}}^{i} (\sigma_{1}, \sigma_{2}) = \Ext_{\tilde{Z} \tilde{T}^{2}}^{i} (\sigma_{1}, \chi_{2}).
\]
Since $\sigma_{1}$ restricted to the abelian subgroup $\tilde{Z} \tilde{T}^{2}$ 
is a sum of
characters, we are reduced to the following claim whose proof we leave to the reader. (Our application
of the claim below to the lemma above will involve $A = \tilde{Z} \tilde{T}^{2}$, and $C=\tilde{Z}^2$.)

{\bf Claim:} If $A$ is an abelian group with $C$ a subgroup of $A$ such that $A/C$ is of the form $F \times \Z$,
for $F$ a pro-finite group. 
Then for characters $\chi_{1}$ and  $\chi_{2}$ of $A$ with $\chi_{1}|_{C} = \chi_{2}|_{C}$ one has
$\Hom_{A} (\chi_{1}, \chi_{2}) \cong \Ext_{A}^{1} (\chi_{1}, \chi_{2})$, and
$\Ext_{A}^{i} (\chi_{1}, \chi_{2}) =0$ for $i \geq 2$
where $\Ext_{A}^{i}$ is calculated in the category of representations of $A$ whose restriction to $C$ is $\chi_{1}|_{C} = \chi_{2}|_{C}$. \end{proof}

\begin{proposition} 
 Let $\pi$ and $\pi'$ be the two sub-quotients of a genuine reducible principal series  representation $\Ind_{\tilde{B}}^{\tilde{G}} \tau$, and $\sigma$ an irreducible representation of $\tilde{T}$ with
 $\omega_{\sigma}|_{\tilde{F}^{\times 2}} = \omega_{\pi}$. Then we have:

\[
\dim \Hom_{\tilde{T}}(\pi, \sigma) + \dim \Hom_{\tilde{T}}(\pi', \sigma) 
= [F^{\times} : F^{\times 2}],
\]
except if $\sigma$ is either $\pi_N$ or $\pi'_N$. 
\end{proposition}

\begin{proof} 
The conclusion of the proposition follows from the exact sequence of Kirillov theory,
\[
0 \rightarrow \mathcal{S}(F^{\times}, \pi_{N, \psi}) \rightarrow \mathtt{K}(\pi) \rightarrow \pi_{N} \rightarrow 0,
\]
together with Lemma \ref{extension lemma for tilde T}.
\end{proof}

\begin{remark} In the previous proposition, we do not know the exact value of
$\dim \Hom_{\tilde{T}}(\pi, \pi_N)$  
which for all we know at the moment may take any of the two values
$\dim \pi_{N,\psi} $ or  $\dim \pi_{N,\psi} +1$; similarly for $\dim \Hom_{\tilde{T}}(\pi', \pi'_N)$. However, since 
$\dim \pi_{N,\psi} $ is either $[F^\times: F^{\times 2}] -1$, or 1,  this does not affect the conclusion of our main theorem 1.1.
\end{remark}  
  
\section{Correspondence $P: \Irrep_{sc}(\widetilde{\SL}_{2}(F)) \rightarrow \Irrep_{sc}({\SL}_{2}(F))$} \label{correspondence}
Let $\Irrep_{sc}(\widetilde{\SL}_{2}(F))$ denote the set of isomorphism classes of irreducible genuine supercuspidal representations of $\widetilde{\SL}_{2}(F)$ and $\Irrep_{sc}(\SL_{2}(F))$, 
the set of isomorphism classes of irreducible supercuspidal representations of $\SL_{2}(F)$.
Assuming that the residue characteristic $p$ of the field $F$ is odd,
we shall define a natural correspondence from $\Irrep_{sc}(\widetilde{\SL}_{2}(F))$ to $\Irrep_{sc}({\SL}_{2}(F))$. 
This correspondence will allow us to transfer a question on representations 
on the covering group $\widetilde{\SL}_{2}(F)$ to a similar question on the linear group $\SL_{2}(F)$.
Let $\mathcal{O}_{F}$ denote the ring of integers of $F$ and $\varpi$ a uniformizer in $\mathcal{O}_{F}$. 

We recall the following well-known result.

\begin{lemma} \cite{Kub69} \label{splitting of K}
Assume that the residue characteristic of $F$ is odd. 
Let $K$ be a maximal compact subgroup of $\SL_{2}(F)$.
Then the covering $\widetilde{\SL}_{2}(F)$ of $\SL_{2}(F)$ splits when restricted to the subgroup $K$.
\end{lemma}

Recall that there are two conjugacy classes of maximal compact subgroups of $\SL_{2}(F)$ which can be represented by 
\[
K_{1} = \SL_{2}(\mathcal{O}_{F}) \text{ and } K_{2} = \left( \begin{matrix} \varpi & 0 \\ 0 & 1 \end{matrix} \right) K_{1} \left( \begin{matrix} \varpi^{-1} & 0 \\ 0 & 1 \end{matrix} \right).
\]
We recall the following result due to Manderscheid.

\begin{proposition}\cite[~Theorem~1.3]{Man84} \label{compact correpondence}
Any irreducible supercuspidal representation of $\widetilde{\SL}_{2}(F)$ can be obtained as an induced representation 
from an irreducible finite dimensional representation of either $\tilde{K}_{1}$ or $\tilde{K}_{2}$.
\end{proposition}

\begin{proposition}
Let $K$ be a maximal compact subgroup of $\SL_{2}(F)$. 
There is a natural bijection between $\Irrep(K)$, the set of isomorphism classes of 
irreducible representations of $K$, and $\Irrep_{gen}(\tilde{K})$ the set of isomorphism classes of irreducible genuine representations of $\tilde{K}$, 
\[
\Irrep_{gen}(\tilde{K}) \longleftrightarrow \Irrep(K).
\] 
\end{proposition}

\begin{proof}
Using a  splitting $s : K \hookrightarrow \tilde{G}$ given by Lemma \ref{splitting of K},  
fix an isomorphism $\tilde{K} \cong K \times_{s} \mu_{2}$. 
Observe that any two splittings $s_1,s_2 : 
K \hookrightarrow \tilde{G}$ differ by a homomorphism $K \rightarrow \mu_{2}$. It is easy to see that $K$ has no  nontrivial
characters of order 2 unless the residue field of $F$ has order 2,3. Since we are only considering odd residue characteristic 
anyway, and since it can be checked that the only nontrivial character of 
$\SL_2({\mathbb F}_3)$ has  order 3, there is a 
unique splitting in all the cases we are considering.

The isomorphism $\tilde{K} \cong K \times_{s} \mu_{2}$
defines a bijection between the set of isomorphism classes of 
irreducible  representations of $K$, and irreducible  genuine representations of $\tilde{K}$. 
Since there is a unique splitting over any maximal compact subgroup $\SL_{2}(F)$, the bijection between 
irreducible representations of $K$ and  irreducible genuine representations of $\tilde{K}$ is canonical. \end{proof}

\begin{theorem} \label{SL2 correspondence}
Let $\Irrep_{sc}(\widetilde{\SL}_{2}(F))$ be the set of equivalence classes of irreducible admissible genuine supercuspidal representations of $\widetilde{\SL}_{2}(F)$ and $\Irrep_{sc}(\SL_{2}(F))$, the set of equivalence classes of irreducible admissible supercuspidal representations of $\SL_{2}(F)$.
There is a natural bijection between 
\[
P : \Irrep_{sc}(\widetilde{\SL}_{2}(F)) \rightarrow \Irrep_{sc}({\SL}_{2}(F)).
\]
\end{theorem}

\begin{proof}
Let $\tilde{\pi}$ be an irreducible admissible supercuspidal representation of $\widetilde{\SL}_{2}(F)$.
By the work of Manderscheid, Proposition \ref{compact correpondence}, $\tilde{\pi}$ is isomorphic to an induced representation $\ind_{\tilde{K}}^{\tilde{G}} \tilde{\sigma}$ which is induced from an irreducible representation $\tilde{\sigma}$ of a maximal compact subgroup $\tilde{K}$ where $K$ is either $K_{1}$ or $K_{2}$.
Let $\sigma \in \Irrep(K)$ which corresponds to $\tilde{\sigma}$ in the manner described in Proposition \ref{compact correpondence}, i.e., under the isomorphism $\widetilde{K} = K \times \Z/2$, $\tilde{\sigma} = \sigma \times {sign}$.

We claim that $\pi := \ind_{K}^{G} \sigma$ is an irreducible admissible supercuspidal representation of $G$.
Given this claim, we define 
\[
P( \tilde{\pi}) = \pi.
\]
It is known that if $\pi$ is irreducible then it is also supercuspidal. 
So we shall only prove that $\pi$ is irreducible. 
By \cite[~Theorem~3.11.4, and remark 1 following it]{BH06} it suffices to prove that  $\Hom_{G} (\pi, \pi) = \C$ where we shorten the notation and write $G$ for $\SL_{2}(F)$. (We thank Sandeep Varma for this reference.)
By Mackey theory, we have
\[
\Hom_{G} (\pi, \pi) = \C \oplus \left( \bigoplus_{1 \neq g \in K \backslash G / K} \Hom_{K \cap K^{g}} (\sigma|_{K \cap K^{g}}, \sigma^{g}|_{K \cap K^{g}}) \right)
\]
Note that under the natural map $\widetilde{G}\rightarrow G$, 
the set of double cosets $\tilde{K} \backslash \tilde{G} / \tilde{K}$ is in bijection with the set of 
double cosets $K \backslash G/K$. 
For $g \in K \backslash G/K$, if we prove
\[
\Hom_{\tilde{K} \cap \tilde{K}^{g}} (\tilde{\sigma}|_{\tilde{K} \cap \tilde{K}^{g}}, \tilde{\sigma}^{g}|_{\tilde{K} \cap \tilde{K}^{g}})  \cong \Hom_{K \cap K^{g}} (\sigma|_{K \cap K^{g}}, \sigma^{g}|_{K \cap K^{g}})
\]
then theorem will follow, because the irreduciblity of the compact induction for $\tilde{\pi} = \ind_{\tilde{K}}^{\tilde{G}} \tilde{\sigma}$ implies that for $1 \neq g \in K \backslash G/K$, the space $\Hom_{\tilde{K} \cap \tilde{K}^{g}} (\tilde{\sigma}|_{\tilde{K} \cap \tilde{K}^{g}}, \tilde{\sigma}^{g}|_{\tilde{K} \cap \tilde{K}^{g}}) = 0$.
\end{proof}
\begin{lemma} \label{Mackey theory for correspondence}
There is an isomorphism
 \[
\Hom_{\tilde{K} \cap \tilde{K}^{g}} (\tilde{\sigma}|_{\tilde{K} \cap \tilde{K}^{g}}, \tilde{\sigma}^{g}|_{\tilde{K} \cap \tilde{K}^{g}})  \cong \Hom_{K \cap K^{g}} (\sigma|_{K \cap K^{g}}, \sigma^{g}|_{K \cap K^{g}}).
 \]
 In particular, both the spaces are simultaneously zero or non-zero.
\end{lemma}
\begin{proof}
Write $\tilde{K} \cong K \times_{s} \mu_{2}$ to emphasize the dependence of the isomorphism on the splitting $s$. 
Recall $\tilde{\sigma} = \sigma \otimes {\rm sign}$.
Note that any representative of the double coset $g \in K \backslash G /K$ can be regarded as a representative in $\tilde{K} \backslash \tilde{G} / \tilde{K}$. 

If we know that the two isomorphisms $\tilde{K} \cap \tilde{K}^{g} \xrightarrow{\sim} (K \cap K^{g}) \times_{s} \mu_{2}$ and $\tilde{K} \cap \tilde{K}^{g} \xrightarrow{\sim} (K \cap K^{g}) \times_{s^{g}} \mu_{2}$ are the same then we have the following
\[
\tilde{\sigma}|_{\tilde{K} \cap \tilde{K}^{g}} \cong \sigma|_{K \cap K^{g}} \otimes_{s} {\rm sign} \text{ and } \tilde{\sigma}^{g}|_{\tilde{K} \cap \tilde{K}^{g}} \cong \sigma^{g}|_{K \cap K^{g}} \otimes_{s^{g}} {\rm sign}.
\]
Thus we have
\[
\begin{array}{lcl}
 \Hom_{\tilde{K} \cap \tilde{K}^{g}} (\tilde{\sigma}|_{\tilde{K} \cap \tilde{K}^{g}}, \tilde{\sigma}^{g}|_{\tilde{K} \cap \tilde{K}^{g}}) & \cong & \Hom_{(K \cap K^{g}) \times \mu_{2}} (\sigma|_{K \cap K^{g}} \otimes {\rm sign}, \sigma^{g}|_{K \cap K^{g}} \otimes {\rm sign}) \\
 & \cong  & \Hom_{K \cap K^{g}} (\sigma|_{K \cap K^{g}}, \sigma^{g}|_{K \cap K^{g}}).
\end{array} 
\]
It remains to prove the following innocuous looking, but crucial lemma.
\end{proof}

The following lemma  uses that the covering group is the
two fold cover of $\SL_2(F)$ since it crucially uses the fact that the inverse image in 
$\widetilde{\SL}_{2}(F)$ 
of the split torus in ${\SL}_{2}(F)$ is abelian, a property which is shared by all metaplectic covers, i.e.,
two fold covers of $\Sp_{2n}(F)$, but is not shared by general covers of general reductive groups. 

\begin{lemma} \label{invariant splitting}
Let $s: K \hookrightarrow \widetilde{\SL}_{2}(F)$ 
be a splitting and $g \in \SL_{2}(F)$.
Let $s^{g} : K^{g} \hookrightarrow \widetilde{\SL}_{2}(F)$ be the splitting of $K^{g}$ given by $s^{g}(k^{g}) = s(k)^{\tilde{g}}$ where $\tilde{g}$ is any lift of $g$ in $\widetilde{\SL}_{2}(F)$.
Then
\[
s|_{K \cap K^{g}} = s^{g}|_{K \cap K^{g}}.
\]
\end{lemma}
\begin{proof}
Note that:\\
\noindent (1) It is enough to prove the lemma for $K=\SL_{2}(\mathcal{O}_{F})$. \\
\noindent (2) It is enough to prove the lemma for $g \in \SL_{2}(F)$ which are a set of 
representatives of the double cosets of $K$ in $\SL_{2}(F)$. These representatives of the double coset of $K \backslash \SL_{2}(F)/ K$ can be taken to be $\underline{a}_{n} := \left( \begin{matrix} \varpi^{n} & 0 \\ 0 & \varpi^{-n} \end{matrix} \right)$ for $n \in \Z_{\geq 0}$. \\
Note that the restriction of the two splittings $s$ and $s^{g}$ on $K \cap K^{g}$ differ by a character of $K \cap K^{g}$ with values in $\{ \pm 1 \}$, i.e., a quadratic character.
Our aim is to prove that this character is trivial.
This character on $K \cap K^{g}$ is given by
\[
k \mapsto s(k^{-1}) s^g(k)=s(k)^{-1} \tilde{g} s(k^{g^{-1}}) \tilde{g}^{-1} \in \{ \pm 1 \}.
\]
Let us write 
\[
\Gamma_{0}(m) := \left\{ \left( \begin{matrix} a & b \\ c & d \end{matrix} \right) \in \SL_{2}(\mathcal{O}_{F}) : c \in \varpi^{m} \mathcal{O}_{F} \right\}  \text{ if } m \in \Z_{\geq 0} \\
\]
Then, for $g = \underline{a}_{n}$, the intersection $K \cap K^{g} = \Gamma_{0}(\varpi^{2n})$.
For $n=0$, there is nothing to prove. 
Now we assume $n \neq 0$ and, then $\Gamma_{0}(\varpi^{2n})$ 
has a normal pro-$p$ subgroup $\Gamma_{1}(\varpi^{2n})$ with quotient isomorphic to $(\mathcal{O}_{F}/\varpi^{2n})^{\times}$.
Since $p \neq 2$, a quadratic character of $K \cap K^{g} = \Gamma_{0}(\varpi^{2n})$ will factor through a quadratic character of $(\mathcal{O}_{F}/\varpi^{2n})^{\times}$. 
%This group has unique non-trivial quadratic character.
Note that  this quadratic character of $\Gamma_{0}(\varpi^{2n})$ 
is trivial if it is trivial on the diagonal elements of  $\Gamma_{0}(\varpi^{2n})$.
Since the inverse image of the diagonal torus of $\SL_{2}(F)$ is abelian, the element 
$s(k)^{-1} \tilde{g} s(k^{g^{-1}}) \tilde{g}^{-1} = 
s(k)^{-1} \tilde{g} s(k) \tilde{g}^{-1} $ is trivial for all diagonal $k$.
Therefore, the map $k \mapsto s(k)^{-1} \tilde{g} s(k^{g^{-1}}) \tilde{g}^{-1}$ is trivial.
\end{proof}

\begin{remark}
Since splitting
$s: K \hookrightarrow \widetilde{\SL}_{2}(F)$ is unique, there is a natural way to write 
$\tilde{K} = s(K) \times \{\pm 1 \} = K \times \{ \pm 1 \}$. 
Similarly, $\tilde{K}^g = K^g \times \{ \pm 1 \}$.  An equivalent way to state the previous lemma 
would be that inside $\widetilde{\SL}_2(F)$, 
$( K \times \{ \pm 1 \} ) \cap (K^g \times \{ \pm 1 \}) = (K \cap K^g) \times \{\pm 1 \}$. 
\end{remark}

The correspondence defined in Theorem \ref{SL2 correspondence} has the following striking
property.

\begin{proposition} \label{same K-types}
Let $\tilde{\pi}$ be an irreducible supercuspidal
representation of $\widetilde{\SL}_{2}(F)$ and $\pi = P(\tilde{\pi})$ be the corresponding
supercuspidal representation of $\SL_2(F)$.
Then
\[
\tilde{\pi}|_{K} \cong \pi|_{K}.
\]
\end{proposition}

\begin{remark} As mentioned 
before Lemma \ref{invariant splitting},
all the results of this section (in particular, Theorem \ref{SL2 correspondence} and Proposition \ref{same K-types}) are valid for
the two fold metaplectic cover of $\Sp_{2n}(F)$ for $F$ of odd residue characteristic.
\end{remark}

\begin{remark} If $p \not = 2$, the theorem on compact induction of irreducible supercuspidal representations of $\SL_2(F)$ as 
$\ind_K^{\SL_2(F)} (\sigma)$ 
allows one to construct an irreducible supercuspidal representation of $\widetilde{\SL}_2(F)$ as
$\ind_{K \times \mu_{2}}^{\widetilde{\SL}_2(F)}(\sigma \otimes {\rm sign})$. 
It is natural to expect that this way we have constructed
all irreducible supercuspidal representation of $\widetilde{\SL}_2(F)$ --- which indeed would be a consequence of the theorem of Manderscheid, i.e. Proposition \ref {compact correpondence} --- although we would like to think that it is a consequence of generalities 
(some kind of Plancherel theorem because their numbers and formal degrees are the same). The advantage of this method would be that it would be a much more
general method of proving the theorem on compact induction of irreducible supercuspidal representations of other covering
 groups such as the two fold cover of $\Sp_{2n}(F)$. (It is known that the maximal reductive quotient of any maximal
compact subgroup of $\Sp_{2n}(F)$ 
is of the form $\Sp_{2\ell} \times \Sp_{2 m}$, in particular is simply connected, so the metaplectic covering of 
$\Sp_{2n}(F)$ 
when restricted to any maximal compact subgroup of $\Sp_{2n}(F)$ splits.) To be sure, our argument on irreducibility 
of $\ind_{K \times \mu_{2}}^{\widetilde{\SL}_2(F)}(\sigma \otimes {\rm sign})$ starting with the irreducible representation 
$\ind_K^{\SL_2(F)} (\sigma)$ of $\SL_2(F)$ works only for $K$ a hyperspecial maximal compact subgroup of $\Sp_{2n}(F)$.
\end{remark}
\section{Restriction of supercuspidal representations}

\subsection{Restriction of supercuspidal representations of $\widetilde{\SL}_{2}(F)$ to $\tilde{E}^{1}$} \label{transfer from metaplectic SL2 to SL2}
In this subsection we study the restriction of an irreducible genuine supercuspidal representation of $\widetilde{\SL}_{2}(F)$ to a non-split torus. 
For any quadratic field extension $E/F$, let $E^{1} := \{ x \in E^{\times} : \Nm x=1 \}$ where $\Nm $ denotes the norm of the quadratic extension.
We fix an embedding $E^{1} \hookrightarrow \SL_{2}(F)$ and we write $E^1$ for the non-split torus of $\SL_{2}(F)$ determined by the field extension $E/F$. 
Let $\tilde{E}^{1}$ denote the inverse image of $E^1$ in the twofold cover $\widetilde{\SL}_{2}(F)$.
\begin{lemma}
The group $\tilde{E}^{1}$ is abelian.
\end{lemma}
\begin{proof} We already know that $E^1 \subset F^\times E^{\times 2}$, and $\tilde{F}^\times \tilde{E}^{\times 2}$ is a maximal abelian subgroup of $\tilde{E}^\times$, so the lemma follows.
\end{proof}

Given a quadratic extension $E/F$, $E^\times$ operates on the two dimensional vector space $E$ over $F$, with $E^1$
leaving stable the maximal compact subring of $E$, thus $E^1 \subset K$, for $K$ a maximal compact subgroup of 
$\SL_2(F)$. 
For $p \neq 2$,  we know that the twofold cover $\widetilde{\SL}_{2}(F)$ splits over $K$ and hence over 
$E^{1}$ giving rise to an isomorphism $\tilde{E}^{1} \cong E^{1} \times \{ \pm 1 \}$.
For any genuine character $\tilde{\nu}$ of $\tilde{E}^{1}$, we associate a character $\nu$ of $E^{1}$ such that $\tilde{\nu} = \nu \otimes {\rm sign}$.

The following result is a consequence of Proposition \ref{same K-types}.

\begin{proposition} \label{covering SL2 to linear SL2}
Let the residue characteristic of $F$ be odd. 
Let $\tilde{\pi}$ be an irreducible genuine supercuspidal representation of $\widetilde{\SL}_{2}(F)$ and $\tilde{\nu}$ a genuine character of $\tilde{E}^{1}$.
Let $\pi = P(\tilde{\pi})$ and $\nu$ the corresponding character of $E^{1}$.
Then there is a natural isomorphism:
\[
\Hom_{\tilde{E}^{1}} (\tilde{\pi}, \tilde{\nu}) \cong \Hom_{E^{1}} (\pi, \nu).
\]
\end{proposition}

In the odd residue characteristic case, Proposition \ref{covering SL2 to linear SL2} enables one to transfer the question of restriction of a supercuspidal representation of $\widetilde{\SL}_{2}(F)$ to $\tilde{E}^{1}$ to a similar question of restriction of a supercuspidal representation of $\SL_{2}(F)$ to $E^{1}$.

\subsection{Restriction of supercuspidal representations of $\widetilde{\GL}_{2}(F)$ to $\tilde{E}^{\times}$}

In 
Proposition \ref{covering SL2 to linear SL2}, we have transferred the restriction problem on covering groups to another restriction problem on linear groups where it is better understood.
Since we are interested in understanding the restriction of an irreducible genuine supercuspidal representation of $\widetilde{\GL}_{2}(F)$ to $\tilde{E}^{\times}$, we  now transfer this question to a related question of restriction of an irreducible genuine supercuspidal representation of $\widetilde{\SL}_{2}(F)$ to $\tilde{E}^{1}$.
We do this without any assumption on the residue characteristic of $F$.

\begin{proposition} \label{metaplectic GL2 to metaplectic SL2}
Let $\tilde{\sigma}$ be an irreducible admissible genuine representation of $\widetilde{\SL}_{2}(F)$ and $\tilde{\nu}$ a genuine character of $\tilde{E}^{1}$.
Let $\mu$ be a character of $\tilde{Z}$ such that the restriction of $\mu$ to the center of $\widetilde{\SL}_{2}(F)$ is the central character of $\tilde{\sigma}$ and 
\[
\tilde{\pi} = \ind_{\tilde{Z} \widetilde{\SL}_{2}(F)}^{\widetilde{\GL}_{2}(F)} \mu \tilde{\sigma} 
\cong \ind_{\tilde{Z} \widetilde{\SL}_{2}(F)}^{\widetilde{\GL}_{2}(F)} \mu^{a} \tilde{\sigma}^{a}.
\]
Let $\lambda$ be a genuine character of $\tilde{F}^{\times} \tilde{E}^{\times 2}$ such that $\lambda|_{\tilde{F}^{\times}} = \mu$ and $\lambda|_{\tilde{E}^{1}} = \tilde{\nu}$. 
Let us write $\tilde{\chi} = \ind_{\tilde{F}^{\times} \tilde{E}^{\times 2}}^{\tilde{E}^{\times}} \lambda$.
Then
\[
\Hom_{\tilde{E}^{\times}} (\tilde{\pi}, \tilde{\chi}) \cong \Hom_{\tilde{E}^{1}} (\tilde{\sigma}, \tilde{\nu}).
\]
\end{proposition}
\begin{proof}
First observe that: ${F}^{\times} {E}^{\times 2} = {F}^{\times} {E}^{1}$
(since for $e \in E^{\times}$, $e^{2} = (e \bar{e})(\frac{e}{\bar{e}}) \in F^{\times} E^{1}$, so 
${F}^{\times} {E}^{\times 2} \subset {F}^{\times} {E}^{1}$; further,
 $\frac{e}{\bar{e}} = \frac{e^2}{e \bar{e}} \in F^{\times} E^{\times 2}$, so 
${F}^{\times} {E}^{\times 2} \supset {F}^{\times} {E}^{1}$). 
Therefore $\tilde{F}^{\times} \tilde{E}^{\times 2}= {F}^{\times} {E}^{1} \subseteq \tilde{Z} \widetilde{\SL}_{2}(F)$.
From \cite{PP16},
recall  the following
\[
\tilde{\pi}|_{\tilde{Z} \widetilde{\SL}_{2}(F)} \cong \bigoplus_{a \in F^{\times}/ F^{\times 2}} \mu^{a} \tilde{\sigma}^{a}.
\]
Using Frobenius reciprocity, we get
\[
\begin{array}{lcl}
\Hom_{\tilde{E}^{\times}} (\tilde{\pi}, \tilde{\chi}) & \cong & \Hom_{\tilde{F}^{\times} \tilde{E}^{\times 2}} (\tilde{\pi}, \lambda) \\
 &=& \bigoplus_{a \in F^{\times}/F^{\times 2}} \Hom_{\tilde{F}^{\times} \tilde{E}^{1}} (\mu^{a} \tilde{\sigma}^{a}, \lambda).
\end{array}
\]
Recall that $\mu = \mu^{a}$ if and only if $a \in F^{\times 2}$. 
We are assuming that $\lambda|_{\tilde{F}^{\times}} = \mu$ and $\lambda|_{\tilde{E}^{1}} = \tilde{\nu}$, therefore we get
\[
\begin{array}{lcl}
\Hom_{\tilde{E}^{\times}} (\tilde{\pi}, \tilde{\chi}) 
& \cong & \Hom_{\tilde{F}^{\times} \tilde{E}^{1}} (\mu \tilde{\sigma}, \lambda) \\
&=& \Hom_{\tilde{E}^{1}} (\tilde{\sigma}, \tilde{\nu}). 
\end{array} 
\]
This completes the proof of the proposition.
\end{proof}

\begin{remark}
It can be easily seen that 
for a given irreducible admissible genuine representation $\tilde{\pi}$ of $\widetilde{\GL}_{2}(F)$ 
and an irreducible genuine representation $\tilde{\chi}$ of $\tilde{E}^{\times}$ 
with $\Hom_{\tilde{E}^{\times}} (\tilde{\pi}, \tilde{\chi}) \not = 0$, 
one can make suitable choices for an irreducible genuine supercuspidal representation $\tilde{\sigma}$ of $\widetilde{\SL}_{2}(F)$ and a genuine character $\tilde{\nu}$ of $\tilde{E}^{1}$ which satisfies 
the conditions in Proposition \ref{metaplectic GL2 to metaplectic SL2}. It follows from \cite{PP16} 
that any
irreducible genuine representation $\tilde{\pi}$ of $\tilde{\GL}_2(F)$ is of the form 
$\tilde{\pi} = \ind_{\tilde{Z} \widetilde{\SL}_{2}(F)}^{\widetilde{\GL}_{2}(F)} \mu \tilde{\sigma}$ used in proposition \ref{metaplectic GL2 to metaplectic SL2}.

\end{remark}

\end{document}